\documentclass[preprint,12pt]{imsart}

\RequirePackage[OT1]{fontenc}
\RequirePackage{amsthm,amsmath}
\RequirePackage[numbers]{natbib}
\RequirePackage[colorlinks,citecolor=blue,urlcolor=blue]{hyperref}

\usepackage{amscd,amsfonts,amssymb,amsmath,latexsym,array,hhline,xcolor,graphicx}
\usepackage{float}
\usepackage{appendix}
\usepackage{booktabs}
\usepackage{caption}
\usepackage{epstopdf}
\usepackage{bbm}

\newcommand\eL{{\mathbb{L}^{0}}}
\newcommand\eD{{\mathbb{D}}}
\newcommand\eE{{\mathbb{E}}}

\newcommand\cF{{\cal F}}

\newcommand\cL{{\cal L}}
\newcommand\cB{{\cal B}}
\newcommand\cH{{\cal H}}

\newcommand\cM{{\cal M}}

\newcommand\R{{\bf R}}
\newcommand\Q{{\bf Q}}

\newcommand\E{\mathbb{E}}
\newcommand\N{\mathbb{N}}

\newcommand\e{{\varepsilon}}
\newcommand\red{\color{red}}

\newtheorem{theo}{Theorem}[section]
\newtheorem{prop}[theo]{Proposition}
\newtheorem{lemm}[theo]{Lemma}
\newtheorem{defi}[theo]{Definition}
\newtheorem{ex}[theo]{Example}

\newtheorem{coro}[theo]{Corollary}
\newtheorem{rem}[theo]{Remark}

\newcommand\beq{\begin{equation}}
\newcommand\eeq{\end{equation}}
\newcommand\beqa{\begin{equation*}}
\newcommand\eeqa{\end{equation*}}
\newcommand\bea{\begin{eqnarray}}
\newcommand\eea{\end{eqnarray}}
\newcommand\bean{\begin{eqnarray*}}
\newcommand\eean{\end{eqnarray*}}

\newcommand\re{\color{red}}

\DeclareMathOperator{\essinf}{ess\;inf}
\DeclareMathOperator{\esssup}{ess\;sup}

\DeclareMathOperator{\csupp}{{\rm c-supp}}

\begin{document}

\begin{frontmatter}

\title{Conditional indicators }\bigskip

\author[A1]{ Dorsaf CHERIF}
\author[A2]{ Emmanuel LEPINETTE}

\address[A1]{ Faculty of Sciences of Tunis, Tunisia.\\
Email: dorsaf-cherif@hotmail.fr
}

\address[A2]{ Ceremade, UMR  CNRS 7534,  Paris Dauphine University, PSL,\\ Place du Mar\'echal De Lattre De Tassigny, 
75775 Paris cedex 16, France.\\
Email: emmanuel.lepinette@ceremade.dauphine.fr
}

\begin{abstract} In this paper, we introduce a large class of (so-called) conditional indicators, on a complete probability space
with respect to a sub $\sigma$-algebra. A conditional indicator is a positive mapping, which is not necessary linear, but may share common features with the conditional expectation, such as the tower property or the projection property. Several characterizations are formulated.  Beyond the definitions, we provide some non trivial examples that are used in finance and may inspire new developments in the theory of operators on Riesz spaces.   
 \end{abstract}

\begin{keyword} Positive operator; conditional expectation; tower-property; projection; conditional risk-measure; stochastic basis; mathematical finance\end{keyword}

\end{frontmatter}
\section{Introduction}

In mathematical finance, the positive expectation operator and, more generally, the conditional expectation operator, is certainly the indicator the most used by the practitioners.  It provides the best estimation $E(X)$, say today, of  any future wealth or price $X$, modeled as a random variable, that is only revealed at some horizon date. Actually, there exists a large variety of indicators that are used in statistics, economics but, also, in  finance, in order to control for example the risk of financial strategies. 

The conditional expectation is the key tool when estimating the  portfolio process replicating a contingent claim in a complete financial market model under the usual no-arbitrage condition, see \cite{FollS}, \cite{DelSch05}. Under this condition, the price process is a martingale under the so-called risk-neutral probability measure, which is fundamental to identify the unique replicating portfolio process from its terminal value, see  \cite{Schal99}. This notion of martingale is generally defined with respect to the expectation operator, which is made possible by the well-known  tower and projection properties. But we may find generalizations to other operators, such as in \cite{BCJ}. Actually, the expectation operator appears naturally in the classical theory because of the financial models themselves that are conic by definition. In particular, the no-arbitrage condition which is imposed appears to be equivalent to the existence of a risk-neutral probability measure, by virtue of the Hahn-Banach separation theorem with respect to the $\sigma(L^\infty,L^1)$ weak  topology, see for example \cite{DMW} and \cite{HK79} in discrete time. 

Nevertheless, as soon as we consider more realistic financial markets with transaction costs, the models are not necessary conic and, worst, they are not convex if there are fixed costs, see \cite{LT} . In that case, the usual arguments derived from the standard duality of the convex analysis, see \cite{PenMOR}, can not be used. In the recent papers \cite{CL} and \cite{EL1}, a new approach is proposed. Not only there is no need to impose a no-arbitrage condition which is, in general, difficult to verify in practice but it is possible to compute numerically the super-hedging  prices  backwards thanks to new results on random optimization, see \cite{EL}. To do so, the fundamental operator we use (called indicator in this paper as it is not linear)  is the conditional essential supremum, both with its dual indicator, i.e. the conditional essential infimum, see \cite{LT}, \cite{LV}, \cite{LZ}. Similarly to the conditional expectation operator, it satisfies the tower property and other common features. In particular,  it is possible to consider martingales w.r.t. such an indicator.

 In our paper, we  define  conditional  ``indicators $I$ " with respect to a sub $\sigma$-algebra $\cH$ as mappings that  map real-valued  random variables $X$ into the subset of $\cH$-measurable random variables. Precisely, $I(X)$ is supposed to belong to the convex hull of the conditional support  of  all possible values of $X$, and satisfies $I(X)=X$ if $X$ is $\cH$-measurable.  This implies that $I$ is positive i.e. $I(X)\ge 0$ if $X\ge 0$.  In finance, a $\sigma$-algebra $\cH$ is generally interpreted as   available information about the market. Then, a conditional indicator is an indicator whose value is updated thanks to the information $\cH$. Classical indicators in finance, but also in statistics, are the quantiles, e.g. the Value At Risk \cite{Lee}, in the domain of risk measures for banking and insurance regulation, see \cite{Delb1} and \cite{DS}.

 In Section \ref{SecDef}, we introduce the main definitions and we give some typical examples of conditional indicators. In Section \ref{SecReg}, we consider and characterize the conditional  ``indicators $I$ " that are said regular, i.e. they  satisfy the property $I(X1_{\cH})=I(X)1_{\cH}$ for all $H\in \cH$. This property is observable in many examples of conditional indicators and used to define the projection property related to the tower property of Section \ref{SecTP}. In Section \ref{SecDI}, the dual of a conditional indicator  is  introduced and  an example  in finance is given.  Natural questions  arise, such as identifying the set of all self-dual conditional indicators. In Section \ref{SecRM}, we make a link between conditional indicators and the conditional risk measures of financial regulation, see also \cite{FP}.   Section \ref{CEI} is  devoted to the conditional expectation defined on the whole space $L^0(\R,\cF)$. We also provide some minimal conditions under which  a conditional indicator is necessarily a  conditional expectation under some absolutely continuous probability measure.\smallskip
 
 In a future work, we hope to take the theory of conditional indicators to the Riesz space setting, thereby generalizing the ideas of \cite{ABCM}, \cite{AL}, \cite{G} and \cite{KLW}.  Some interesting problems are open such as characterizing the linear conditional indicators, studying the indicators satisfying the tower property and the associated notion of martingales but, also, identifying the stochastic indicators which are uniquely defined by the projection property, see a first result in that direction given by Proposition \ref{uniq1}. Actually, the notion of conditional expectation in the field of Riesz spaces and positive operators is very popular and has given  rise to new developments recently, see \cite{AL} for an overview on positive operators and the papers \cite{KLW} and \cite{G} on conditional expectation, among others. Naturally, the concept of martingale has been introduced, see \cite{KLW} and \cite{G}. As mentioned above, some non linear positive operators are also needed in finance and we think that they may inspire interesting problems for the community of people working on positive operators and Riesz spaces, see for example \cite{ABCM}.

\section{Conditional indicators}\label{SecDef}

We consider a probability space $(\Omega,\mathcal{F}, \mathbb{P})$ where the $\sigma$-algebra $\mathcal{F}$ is supposed to be complete with respect to $\mathbb{P}$.  Let $\mathcal{H}$ be a sub $\sigma$-algebra of $\mathcal{F}$ which is also supposed to be complete. In the whole paper, we use the following notations.\smallskip
 
\noindent {\bf Notations and conventions:}\smallskip
 
1) For any $r\in \R$, we adopt the conventions that $ r \pm \infty = \pm \infty,$ $\infty - \infty =0$, and  $\infty + \infty = \infty$ and $0\times \pm \infty=0$. 

 By virtue of our notational conventions, we deduce that $\alpha(a-b)=\alpha a - \alpha b$ for all $\alpha \in \R$ and $a,b\in \overline{\R}$.
 \smallskip
 
 2)   For any subset $G$ of $\R$, we denote by $\eL(G, \cF)$ (resp. $\eL(G, \cH)$) the set of all $\cF$-measurable (resp. $\cH$-measurable) random variables $X$ such that $X(\omega)\in G$ a.s..\smallskip
 
 3) We define the extended real line $\overline{\R}=\R\cup\{-\infty,\infty\}$. \smallskip
 
 We recall the concept of conditional supremum and infimum, see \cite{KS}[Section 5.3.1], \cite{KL}:
 
 \begin{theo} Let $\Gamma$ be a family of $\cF$-measurable random variables with values in $\overline{\R}$ and let $\cH$ be a sub $\sigma$-algebra of $\cF$. There exists a unique $\cH$-measurable random variable $\esssup_{\cH} \Gamma$ such that:
 \begin{itemize}
 \item [1)] $\esssup_{\cH} \Gamma \ge \gamma$, for all   $\gamma\in \Gamma$,
 
 \item [2)] If $\hat \gamma$ is $\cH$-measurable and $\hat \gamma \ge \gamma$, for all   $\gamma\in \Gamma$, then $\hat \gamma \ge \esssup_{\cH} \Gamma$.
  
 \end{itemize}
 
  \end{theo}
 
 Note that $\esssup_{\cH} \Gamma$ is smallest $\cH$-measurable variable that dominates the family $\Gamma$. Symmetrically, we define $\essinf_{\cH} \Gamma:=-\esssup_{\cH}(- \Gamma)$ as the largest $\cH$-measurable variable that is dominated by the family $\Gamma$.

\begin{defi}\label{DefiCI}

Let $\eD_{I}$ be a subset of $\eL(\overline{\R}, \cF)$ containing $0$.
We say that a mapping 
\bean
I_{\cH}: \eD_{I} & \longrightarrow & \eL(\overline{\R}, \cH).\\
 X & \longmapsto & I_{\cH}(X) 
\eean
is a Conditional Indicator  (C.I.) if the following properties hold:
\begin{itemize}
\item[(P1)]  $I_{\mathcal{H}}(X) \in \csupp_{\mathcal{H}}(X):=[\essinf_{\mathcal{H}}(X),\esssup_{\mathcal{H}}(X)]$ a.s. \smallskip
\item[(P2)]  $\mathbb{D}_{I}+\eL(\overline{\R},\cH) \subseteq \mathbb{D}_{I}$.
\end{itemize}

\end{defi}

\begin{rem}
For the sake of simplicity, we write  $I$ instead of $I_{\cH}$ when $\cH$ is fixed without any possible confusion. Note that, for all $X \in \eL(\overline{\R}, \cH)$, $I(X)=X$, i.e. $I$ is  idempotent. We also observe that it is always possible to extend a conditional indicator to the whole set $\eL(\R,\cF)$. Indeed, it suffices to define for example $I(X)=\esssup_{\mathcal{H}}(X)$ for $X\in \eL(\overline{\R}, \cF)\setminus \mathbb{D}_{I}$. In the following, the domain of definition of any conditional indicator is always denoted by $\eD_I$.
\end{rem}

\begin{rem} The natural extension of Definition \ref{DefiCI} to  multi-varied random variables is to suppose that $I_{\mathcal{H}}(X)$ belongs a.s. to the convex hull of the conditional support $\csupp_{\mathcal{H}}(X)$, see the definition in \cite{EL}, which is no more an interval. This possible generalization is an open problem beyond the scope of this paper.
\end{rem}

\begin{lemm}\label{pos}
Let $\mathcal{H}$ be a sub-$\sigma$-algebra of $\mathcal{F}$, and let $I$ be a C.I. w.r.t. $\mathcal{H}$.
Then, $I$ is a positive indicator, i.e. $I(X) \ge 0$, for all $X \in \mathbb{D}_I$ such that $X \ge 0$. 
  	In particular, if $I$ is linear, then $I$ is increasing. 
  	\end{lemm}
	\begin{proof} Let $X \in \mathbb{D}_I^+$.
  	As  $I(X) \in \csupp_{\mathcal{H}}(X)$, then $I_{\mathcal{H}}(X) \geq \mathrm{essinf}_\mathcal{H}(X) \ge 0 
	$ and the conclusion follows.
  	\end{proof}

\begin{defi}
Let $I$ be a C.I. Then, 
\begin{itemize}
\item[1)] $I$ is said increasing if,  for all $X,Y \in \eD_I$ such that  $X \leq Y$, we have  $I(X) \leq I(Y)$.
\item[2)] $I$ is said $\mathcal{H}$-translation invariant if  $I(X+Y_{\cH})=I(X)+Y_{\cH}$ for all $X \in \eD_I$ and $Y_{\cH} \in \eL(\R, \cH)$ such that $X+Y_{\cH}\in \eD_I$. 
\item[3)] $I$ is said $\cH$-positively-homogeneous if, for  every $\alpha_{\cH}  \in \eL(\R_+, \cH)$, we have  $\alpha_{\cH} \eD_I \subset \eD_I$ and for any $X \in \eD_I$, $I(\alpha_{\cH} X)=\alpha_{\cH} I(X)$.
\item[4)] $I$ is said $\cH$-linear  if,  for all   $\alpha_{\cH}  \in \eL(\R, \cH)$, $\alpha_{\cH} \eD_I +\eD_I \subset \eD_I$, and for every $X,Y \in \eD_I$,  $ I(\alpha_{\cH}  X+Y)=\alpha_{\cH} I(X)+I(Y)$.
\item [5)]  If   $I(\lim\sup_n X_n) \geq  \lim\sup_n I(X_n)$ (resp. $I(\lim\inf_n X_n) \leq  \lim\inf_n I(X_n) $), for any sequence $(X_n)_n \in \eD_I$ such that $\lim\sup_n X_n\in \eD_I$ (resp. $\lim\inf_n X_n \in \eD_I$), we say that $I$  satisfies the upper (resp. lower) Fatou property.
\item [6)] $I$ is  said conditionally convex if, for any $\alpha_{\cH}  \in \eL([0,1], \cH)$, we have $\alpha_{\cH} \eD_I +  (1-\alpha_{\cH})  \eD_I \subset \eD_I$ and  for all $X \in \eD_I$,

$$ I(\alpha_{\cH} X + (1-\alpha_{\cH} ) Y)  \le \alpha_{\cH}  I(X)+ (1-\alpha_{\cH})  I(Y).$$
\end{itemize}
\end{defi}

\begin{rem} The conditional expectation operator $I^1(X)=\E(X|\mathcal{H})$  is a well known  example of conditional indicator  which  is  ${\mathcal{H}}$-linear, $\mathcal{H}$-translation invariant and  increasing on $\eD_{I^1}=\mathbb{L}^{1}(\R, \mathcal{F})\cup \eL( \overline{\R},\cH)$, where $\mathbb{L}^{1}(\R, \mathcal{F})$ is  the set of all integrable random variables. \smallskip

The conditional supremum  $I^2(X)=\esssup_{\mathcal{H}}(X)$  is another {\red example} defined on $\eD_{I^2}=\eL( \overline{\R},\cF)$.  Note that $I^2$ is  increasing, $\mathcal{H}$-translation invariant, ${\mathcal{H}}$-positively-homogeneous and sub-additive.\smallskip

If $I:\eD_{I}  \longrightarrow  \mathbb{L}^{0}(\overline{\R}, \cH)$
  	 is  increasing  and satisfies $I(X_{\cH})=X_{\cH}$ for all $X_{\cH} \in \eL( \overline{\R},\mathcal{H})$,  then  $I$ is a C.I. 
 \end{rem}
 
 \begin{lemm}[lower and upper extensions of a conditional indicator]\label{LUECI} Consider a conditional indicator $I$ defined on some domain $\eD_I$ which is monotone. Suppose that $E_{I}$ is a subset of $\eD_I$ containing $\eL( \overline{\R},\cH)$. Let us define:
 \bea \label{ext1} I^L(X):&=&\esssup_{\cH}\{I(Y):~Y\in E_{I}\,{\rm and}\,Y\le X\},\\ \label{ext2}
 I^U(X):&=&\essinf_{\cH}\{I(Y):~Y\in E_{I}\,{\rm and}\,Y\ge X\}.
 \eea
 Then, $ I^L$ and $I^U$ are two monotone conditional indicators defined on $\eL(\overline{\R}, \cF)$ that coincide with $I$ on $E_{I}$ and satisfies $ I^L\le I\le I^U$ on  $\eD_I$. We say that $ I^L$ and $I^U$ are   lower and upper extensions of $I$ on $E_I$. If $E_I+E_I\subseteq E_I$, then $I^L$ and $I^U$ are respectively super and sub-additive. 
 \end{lemm}
  \begin{proof}
First observe that, if $Y\in  E_{I}$ is such that $Y\le X$, then we have $I(Y)\le \esssup_{\cH}(Y)\le  \esssup_{\cH}(X)$ hence, taking the essential supremum, we get that  $I^L(X)\le \esssup_{\cH}(X)$. Moreover, if $X\in \eD_I$, $Y\le X$ implies, by assumption, that $I(
Y)\le I(X)$  hence  $I^L(X)\le I(X) $ for $X\in \eD_I$. On the other hand, as $\essinf_{\cH}(X)\le X$, we deduce that   $I^L(X)\ge \essinf_{\cH}(X)$. At last, if $X_{\cH}\in \eL(\overline{\R},\mathcal{H})$, then $X_{\cH} \in E_{I}$ hence  $I^L(X_{\cH})\ge X_{\cH}$. Moreover,  $I^L(X_{\cH}) \le I(X_{\cH})=X_{\cH}$ so that  $I^L(X_{\cH})=X_{\cH}$. Note that, if $X\in  E_{I}$, then  $I^L(X)\ge I(X)$. As  $I^L(X)\le I(X)$, we conclude that   $I^L(X)= I(X)$. The same types of argument hold for  $I^U$. 
 \end{proof}
 
 \begin{rem}\label{essupIndic} If $(I_k)_{k\in K}$ is a family of conditional indicators w.r.t. the $\sigma$-algebra $\cH$, then $I_1(X)=\essinf_{k\in K}I_k(X)$ and $I_2(X)=\esssup_{k\in K}I_k(X)$ are still conditional indicators w.r.t. to $\cH$ on $\eD_{I_1}=\eD_{I_2}=\cap_{k\in K}\eD_{I_k}$. As Lemma \ref{LUECI} proves the existence of upper and lower extensions,  we then deduce the following corollary.
 
 \end{rem}
 
 \begin{coro}\label{CoroOptExtens} Consider a conditional indicator $I$, w.r.t. the $\sigma$-algebra $\cH$, defined on some domain $\eD_I$,  which is monotone. Suppose that $E_{I}$ is a subset of $\eD_I$ containing $\eL( \overline{\R},\cH)$. There exists a (unique) smallest   conditional indicator $I^+$ (resp. a largest conditional indicator $I^-$) which coincides with $I$ on $E_{I}$ and such that $I^-\le I\le  I^+$ on $\eD_I$.
 \end{coro}
 \begin{proof} By Lemma \ref{LUECI}, there exists  conditional indicators $J,K$, defined on $\eL( \overline{\R},\cF)$ such that $J\le I\le K$ on  $\eD_I$. By Lemma \ref{essupIndic}, it suffices to define the indicator $I^-(X)=\esssup_{k\in K}J_k(X)$ where $(J_k)_{k\in K}$ is the non empty family of conditional indicators  that is dominated by $I$ on $\eD_I$ and coincides with $I$ on $E_I$  and the indicator $I^+(X)=\essinf_{k\in K}K_k(X)$ where $(K_k)_{k\in K}$ is the family of non empty conditional indicators that dominate $I$ on $\eD_I$  and coincides with $I$ on $E_I$.
 \end{proof}

 \section{Regularity}\label{SecReg}
 
\begin{defi} A subset $E$ of $\eL(\overline{\R},\cF)$ is said $\cH$-decomposable if, for all $H\in \cH$, and $X,Y\in E$, we have
$X1_{H}+Y1_{\Omega \setminus H}\in E$.
\end{defi}

\begin{lemm}[Averaging property] \label{HH^c}
Let $\mathcal{H}$ be a sub $\sigma$-algebra of $\mathcal{F}$   and let $I$ be a C.I. w.r.t. $\mathcal{H}$, which is defined on an $\cH$-decomposable subset $\eD_I$. Then, for all  $ X \in \eD_I $ and $ H \in \cH$, $I(X 1_{H})1_{H^c}=0$. Therefore, we have $I(X 1_{H})=I(X 1_{H})1_{H}$.
\end{lemm}
\begin{proof}
Consider $ X \in \mathbb{D}_I $ and $ H \in \mathcal{H}$. As $I(X 1_{H^c}) 1_H \in \csupp_{\mathcal{H}}(X1_{H^c}) 1_H$ and $\csupp_{\mathcal{H}}(X1_{H^c}) 1_H=\csupp_{\mathcal{H}}(X1_{H^c} 1_H)=\{0\} $ by the properties satisfied by the essential infimum and supremum, we get  $I(X 1_{H})1_{H^c}=0$. Therefore, we have the equality  $I(X 1_{H})=I(X 1_{H})1_{H}+I(X 1_{H})1_{H^c}=I(X 1_{H})1_{H}$. 
\end{proof}

\begin{defi}
Consider $\mathcal{H}$  a sub $\sigma$-algebra of $\mathcal{F}$. A conditional indicator $I$ w.r.t. $\mathcal{H}$  is said  regular if $\eD_I$ is $\cH$-decomposable and, for all  $ X ,Y\in \eD_I $ and $ H \in \cH$, we have:

$$X 1_H=Y 1_H \Rightarrow I(X) 1_H= I(Y) 1_H.$$
\end{defi}

The proof of the following lemma is trivial:

\begin{lemm}\label{1H} Let $I$ be a C.I. w.r.t. $\mathcal{H}$, which is defined on an $\cH$-decomposable subset $\eD_I$. The following statements are equivalent.
 \begin{enumerate}

\item  I is regular.
\item $I(X 1_{H})=I(X )1_{H}$, for all  $ X \in \eD_I $ and $ H \in \cH$.
\item  $I(X1_{H}+ Y1_{H^c} )=I(X)1_{H}+I(Y)1_{H^c}$,  for all  $ X,Y \in \eD_I $ and $ H \in \cH$.
\end{enumerate}
\end{lemm}

\begin{prop}
Let $\mathcal{H}$ be  a sub $\sigma$-algebra of $\mathcal{F}$  and let $I$ be a regular and $\cH$-positively-homogeneous conditional indicator w.r.t. $\mathcal{H}$. Then, for all  $ X \in \mathbb{D}_I $ and $ h \in \mathbb{L}^{0}(\R , \mathcal{H})$, we have 
  $$I(h X )=h^+ I(X)+ h^- I(-X).$$
\end{prop}
\begin{proof}
Let  $ X \in \mathbb{D}_I $ and $ h \in \mathbb{L}^{0}(\R , \mathcal{H})$. By Propostion \ref{1H}, we get that
 \bean 
 I(h X )&=&I(hX 1_{\{ h \ge 0\}})+ I(hX1_{\{ h < 0\}})\\
 &=& I(h^+ X)+ I(-h^- X)
 \\
 &=& h^+ I(X)+h^- I(-X) \eean
\end{proof}
\begin{prop}
If a conditional indicator is conditionally convex and $\eD_I$ is $\cH$-decomposable, 
then it is regular.
\end{prop}

\begin{proof}
Let us consider  $H \in \cH$ and  $X\in \cL^0(\R, \cF) $. Since $1_H \in \cL^0([0,1], \cH) $ and $1_{H^c}=1-1_H$, the  conditional convexity of $I$ implies that 
$$I(1_H X) = I(1_H X +1_{H^c} 0 )\le   1_H I(X)+ 1_{H^c}I(0)  =  1_H I(X).$$ 
 Similarly, we have
$$I(X) = I(1_H (1_H X) + 1_{H^c} (1_{H^c} X)) \le 1_HI(1_H X) + 1_{H^c} I(1_{H^c} X).$$
By Lemma \ref{HH^c}, we deduce that $1_H I(X) \le 1_H I(1_H X) =I(1_H X).$ Therefore, $I(1_H X) = 1_H I(X).$
The conclusion follows by Lemma \ref{1H}.
\end{proof}

\begin{lemm}\label{add} Let $I$ be a C.I. w.r.t. the $\sigma$-algebra $\cH$. If $I$ is sub-additive and $\eD_I$ is $\cH$-decomposable, then we have $1_H I(X) \le  I(1_HX)$  for all $H\in \cH$ and $X \in \cL^0(\R, \cF).$ Moreover,  if $I$ is additive, then $I$ is regular. 
\end{lemm}
\begin{proof}
 Consider $H\in \cH$ and $X \in \cL^0(\R, \cF).$ 
 If $I$ is sub-additive, then by Lemma  \ref{HH^c}, we have:
\bean 1_H I(X)&=& 1_H I(1_H X + 1_{H^c} X)\le 1_H I(1_H X) + 1_{H} I(1_{H^c} X)\\
&\le & 1_H I(1_H X)=I(1_H X).\eean
If $I$ is additive, then 
\bean 1_H I(X)&=&  1_H I(1_H X + 1_{H^c} X) = 1_H I(1_H X) + 1_{H} I(1_{H^c} X)\\
&=& 1_H I(1_H X)=I(1_H X).\eean
\end{proof}

\begin{prop}\label{RegExt} Let $I^L$ and $I^U$ be the extensions of $I$ defined by (\ref{ext1}) and (\ref{ext2}) in Lemma \ref{LUECI} when $E_I=\eD_I$. Suppose that $I$ is regular. Then, $I^L$ and $I^U$ are regular.
\end{prop}
\begin{proof}
Let $H\in \cH$ and $X\in \eL(\overline{\R}, \cF)$. Consider any $ Y\le X1_H$ where $Y\in \eD_I$.  Let us define $Z=Y1_{H}+(\essinf_{\cH}X)1_{\Omega\setminus H}$. Then, $Z\le X$ and $Z\in \eD_I$ so that $I(Z)\le I^L(X)$. As $I$ is regular by assumption, we deduce that 
$I(Y)1_{H}+(\essinf_{\cH}X)1_{\Omega\setminus H}\le I^L(X)$. Therefore, $I(Y)1_{H}\le  I^L(X)1_{H}$. Taking the essential supremum, we deduce that $I^L(X1_H)1_{H}\le  I^L(X)1_{H}$. On the other hand, for any $Y\le X$ such that $Y\in \eD_I$, we have $Y1_{H}\le X1_{H}$ and $Y1_{H}\in \eD_I$. Therefore,  $I^L(X1_{H})\ge I(Y1_{H})=I(Y)1_{H}$. Taking the essential supremum, we deduce that  $I^L(X1_{H})\ge I^L(X) 1_{H}$. With the first part of the proof, we deduce that  $I^L(X1_{H})1_{H}=I^L(X) 1_{H}$ for any $X\in \eL(\overline{\R}, \cF)$ and $H\in \cH$. Replacing $X$ by $X1_{H}$ and $H$ by $\Omega\setminus H$, we then deduce the equality
$I^L(X1_{H})1_{\Omega\setminus H}=I^L(X1_{H}1_{\Omega\setminus H})1_{\Omega\setminus H}=I^L(0)1_{\Omega\setminus H}=0$. Therefore, we have  $I^L(X1_{H})=I^L(X1_{H})1_{H}$ hence  $I^L(X1_{H})=I^L(X) 1_{H}$. The conclusion follows by Lemma \ref{1H}. The reasoning is similar for $I^U$.
\end{proof}

\begin{coro}\label{coroExt}Consider a conditional indicator $I$, w.r.t. the $\sigma$-algebra $\cH$, defined on some domain $\eD_I$,  which is monotone and regular. Suppose that $E_{I}$ is a subset of $\eD_I$ containing $\eL( \overline{\R},\cH)$. There exists a (unique) smallest  regular  conditional indicator $I^+$ (resp. a largest regular conditional indicator $I^-$) which coincides with $I$ on $E_{I}$ and such that $I^-\le I\le  I^+$ on $\eD_I$.
 \end{coro}
\begin{proof} It suffices to repeat the proof of Corollary \ref{CoroOptExtens} by restricting the families to the regular indicators. Indeed,  existence holds by Proposition \ref{RegExt}.
\end{proof}

 \begin{lemm} Consider the conditional expectation $I(X)=E(X|\cH)$ for $X\in \eD_I$ where $\eD_I=\eL( \overline{\R},\cH)\cup \mathbb{L}^{1}(\R,\cF)\cup \mathbb{L}^{0}(\R_+,\cF)\cup \mathbb{L}^{0}(\R_-,\cF)$. Suppose that $E_I=\eD_I$, see Corollary \ref{coroExt}. Then, there  exists regular extensions $I^-$ and $I^+$ of the conditional expectation to the whole set $\eL( \overline{\R},\cF)$.
 \end{lemm}
 \begin{proof}
Consider $X=X^+-X^-$, supposed to be integrable, where we recall that  $X^+=\max(X,0)\ge 0$ and  $X^-=-\min(X,0)\ge 0$. Then, $X^+=\lim_n X^n$ where $X^n=X^+\wedge n$, $n\ge 1$, is an increasing sequence of integrable random variables. Then, we get that $X^n-X^+\in E_I=\eD_I$ and $X^n\le X$ hence we get that $I^-(X)\ge I(X^n)$. Taking the limit, we get that $I^-(X)\ge I(X)$ so that $I^-(X)= I(X)$ for all $X=X^+-X^-$ where $X^-$ is integrable. More generally, if $E(X^-|\cH)<\infty$ a.s., by regularity of $I^-$, we get that
 
\bean I^-(X)&=&\sum_{k=0}^\infty I^-(X)1_{\{k\le E(X^-|\cH \}<k+1\}}\\
&=&\sum_{k=0}^\infty I^-(X1_{\{k\le E(X^-|\cH \}<k+1\}})1_{\{k\le E(X^-|\cH \}<k+1\}}.
\eean
As $X1_{\{k\le E(X^-|\cH \}<k+1\}}$ is of the form $X^+-X^-$  where $X^-$ is integrable, we deduce by above that 
\bean I^-(X1_{\{k\le E(X^-|\cH \}<k+1\}})&=&I(X1_{\{k\le E(X^-|\cH \}<k+1\}})\\
&=&(E(X^+|\cH)-E(X^-|\cH))1_{\{k\le E(X^-|\cH \}<k+1\}}\eean
 and, finally  $I^-(X)=E(X^+|\cH)-E(X^-|\cH)$ for every $X$ such that we have $E(X^-|\cH)<\infty$. Similarly, we have $I^+(X)=E(X^+|\cH)-E(X^-|\cH)$ for every $X$ such that $E(X^+|\cH)<\infty$. Note that $I^-$ and $I^+$ are natural extensions of the conditional expectation with the conventions $+\infty -\R=\{+\infty\}$ and $\R-\infty=\{-\infty\}$.
  \end{proof}

\section{Dual indicators}\label{SecDI}	
  
  \begin{defi}
Let $I$ be a C.I. on a  domain  $\mathbb{D}_I$. The dual indicator $I^*$ of $I$ is defined on $\mathbb{D}_{I^*}= -\mathbb{D}_I$  as
 $I^*(X)= -I(-X)$, $X \in \mathbb{D}_{I^*} .$ If $I=I^*$, we say that $I$ is self-dual.
\end{defi}

\begin{prop} The dual $I^*$ of a C.I. $I$ is still a C.I. such that $(I^*)^*=I$  and we have:
\begin{itemize}
\item [1)] If $I $ is monotone, then $I^*$ is monotone.
\item [2)] If $I$ is $\mathcal{H}$-translation invariant, then $I^*$ is $\mathcal{H}$-translation invariant.
\item [3)]  If $I$ is super-linear, then $I^*$ is sub-linear.
\end{itemize}
\end{prop}

\begin{ex}
If $I_1(X)=E(X \vert \mathcal{H})$ is defined for  $X \in \mathbb{D}_{I_1}= \mathbb{L}^{1}(\R, \mathcal{F})$, then $I_{1}^*=I_1$. 
 The indicator $I_2(X)=\esssup_{\mathcal{H}}(X  )$,  $X \in \mathbb{D}_{I_2}=\mathbb{L}^{0}(\R, \mathcal{F})$, admits the dual $I_2^*(X)=\essinf_{\mathcal{H}}(X).$ Let $I$ be any  C.I.. Then, $T= \frac{1}{2}I+ \frac{1}{2}I^*$, defined on $ \mathbb{D}_T=\mathbb{D}_I \cap \mathbb{D}_{I^*}$, is self-dual and is still a C.I.. Reciprocally, any self-dual indicator $T$ is of the form $T= \frac{1}{2}I+ \frac{1}{2}I^*$. Indeed, it suffices to choose $I=T$.

\end{ex}

The proof of the following is left to the readers:

\begin{lemm} Consider the lower and upper conditional indicators $I^L=I^{L(E)}$ and $I^U=I^{U(E)}$ with respect to a subset $E$ of the domain of definition  $\eD_I$ of a conditional indicator $I$, as defined in Lemma \ref{LUECI}. Then, we have the following equalities: $(I^{L(E)})^{*}=(I^{*})^{U(-E)}$ and $(I^{U(E)})^{*}=(I^{*})^{L(-E)}$.
\end{lemm}

 \begin{theo}\label{linear}
Let  $\mathcal{H}$ be a sub $\sigma$-algebra of $\mathcal{F}$, and let $I$ be a C.I. w.r.t. $\mathcal{H}$, defined on a vector space $\mathbb{D}_I$ of $\eL(\overline{\R}, \cF)$ such that $I^*=I$. Then,  $I$ is additive if and only if $I$ is ${\mathcal{H}}$-linear. 
\end{theo}

\begin{proof}
As $I$ is additive, we deduce that  $I$ is regular by lemma \ref{add}, i.e. $I(X 1_{H})=  I(X) 1_H$,
  for all $ X \in \mathbb{D}_I$ and $ H \in \mathcal{H}$. Consider  $ X \in \mathbb{D}_I$ and  $n_H \in \mathbb{L}^{0}(\N, \mathcal{H})$. By additivity, $I(k  X)=kI(X)$, for all $k\in \N$, so that  we have:
     \bean I_{\mathcal{H}}(n_H X)&=& \sum_{k=0}^\infty 1_{\{n_H=k\}} I(n_H X)
     \\ &=& \sum_{k=0}^\infty 1_{\{n_H=k\}} I_{\mathcal{H}}(n_H X 1_{\{n_H=k\}})
     \\&=&\sum_{k=0}^\infty 1_{\{n_H=k\}} I(k  X)
     \\ &=&\sum_{k=0}^\infty 1_{\{n_H=k\}}k I( X)=n_H I( X).
     \eean 
   Also, we get that $I(-n_H X ) = n_H I(-X)= -n_H I^*(X) = -n_H I(X).$ We then deduce that $I(n_H X)=n_H I(X)$ for all $n_H\in \eL(\mathbb{Z}, \mathcal{H})$. Let us consider   $r_H \in \mathbb{L}^{0}(\Q, \mathcal{H})$. By Lemma \ref{Q}, there exists $p_H, q_H  \in \eL(\mathbb{Z}, \mathcal{H}) \times \eL(\N^*, \mathcal{H})$ such that $r_H= p_H/ q_H$. By the first step, we get that
     $$p_H I(X)=I(q_H p_H/q_H X)= q_H I(p_H/q_H X ).$$ 
     Therefore,  $I( r_H X)= r_H I(X).$ Finally, consider  $\alpha \in \mathbb{L}^{0}(\R, \mathcal{H})$. By lemma \ref{R}, there exists two sequences $(r_n)_n , (q_n)_n \in \mathbb{L}^{0}(\Q, \mathcal{H})$ such that $$\alpha-1/n \le r_n \le \alpha \le q_n \le \alpha+1/n . $$
  By the properties above, we deduce that $I$  is increasing. We deduce by the second step that
\bean  &&I(r_n X) \le I(\alpha X) \le I(q_n X),\\
 && r_n I( X) \le I(\alpha X) \le q_n I(X).\eean
 We conclude  that $I( \alpha X)= \alpha I(X)$, as $n\to \infty$.
\end{proof}

\section{Stochastic indicators: tower property and  projection w.r.t. a filtration}\label{SecTP}

Let $(\cF_t)_{t\in [0,T]}$ be a complete filtration,  i.e. a sequence of complete $\sigma$-algebras such that   $\cF_s \subset \cF_t$ for any $s \le t$. Consider   a family $(I_t)_{t\in [0,T]}$ of adapted conditional indicators in the sense that $I_t$ is a conditional indicator w.r.t. $\cF_t$, for every $t\in [0,T]$.  We say that $I=(I_t)_{t\in [0,T]}$ is a stochastic indicator.

\begin{defi}
Consider  a  stochastic indicator $I=(I_t)_{t\in [0,T]}$. We say that $(I_t)_{t\in [0,T]}$ satisfies the tower property if, for any $s \le t$,  $I_t(\eD_{I_s})\subseteq \eD_{I_s} \subseteq \eD_{I_t}$ and
$$ I_s(I_t(X))=I_s(X),\, {\rm for\, all\,} X\in \eD_{I_s} .$$
\end{defi}

\begin{ex} The conditional essential supremum indicator satisfies the tower property. In particular, we have  $$\esssup_{\cF_0}(\esssup_{\cF_t}(X) 1_{F_t}) = \esssup_{\cF_0}(X 1_{F_t}),\, \forall F_t \in \cF_t,\, \forall X\in L^0(\R,\cF).$$
Note that, if a  stochastic indicator $I=(I_t)_{t\in [0,T]}$ satisfies the tower property, so does its dual $I^*=(I_t^*)_{t\in [0,T]}$.
\end{ex}

\begin{defi}
Let $I_0$ be a conditional indicator w.r.t. $\cF_0$ such that $D_{I_0}$ is $\cF_t$-decomposable for every $t\ge 0$. We say that  $Z_t\in \eL(\R,\cF_t)$ satisfies the projection equality at time $t$ w.r.t. $I_0$ if the following condition holds:
\bean  {\rm {\bf Pr}:}\, I_0(X 1_{F_t})=I_0
(Z_t1_{F_t}),\,{\rm for\, all\,} F_t \in \cF_t.\eean

\end{defi}

\begin{prop}\label{uniq1}
 Suppose that the stochastic indicator $I=(I_t)_{t\in [0,T]}$ is such that  $I_0$ is super-additive, $I_t(X)$ satisfies the projection  property at given time $t$ for some $X\in \eD_{I_t}$ and, for every $Y \in \cL^0(\R^+, \cF_T)$, we have $I_0(Y)\le 0$ if and only if $ Y=0$. Then, $I_t(X)$ is the unique  $\cF_t$-measurable random variable satisfying the projection equality  ( {\bf Pr}). 
\end{prop}
\begin{proof}
Suppose that there exists $Z_t \in \cL^0(\R, \cF_t)$ such that, for all $F_t \in \cF_t$,
\bea \label{ID2} I_0(X 1_{F_t})=I_0
(Z_t 1_{F_t}).\eea
 Let us show that $Z_t = I_t(X)$. Take $F_t = \{Z_t > I_t(X)\}$. 
By the projection property  ({\bf Pr}) for $I_t(X)$ and (\ref{ID2}), $ I_0(I_t(X) 1_{F_t})=I_0
(Z_t 1_{F_t}).$ It follows by the super-additivity of $I_0$ that   $I_0((Z_t -I_t(X)) 1_{F_t})\le 0$. Since $(Z_t -I_t(X)) 1_{F_t} \ge 0$ then $(Z_t -I_t(X)) 1_{F_t}=0$ by assumption. So $1_{F_t}=0$ and  $Z_t \le I_t(X)$ a.s.. Analogously $Z_t \ge  I_t(X)$ a.s.. The conclusion follows. 
\end{proof}
\begin{prop}
Suppose that the stochastic indicator $I=(I_t)_{t\in [0,T]}$ is such that  $I_0$ is linear, $I_t(X)$ satisfies the projection  property at time $t$ for all $X \in \cL^0(\R, \cF_T) $ and, for all $Y \in \cL^0(\R^+, \cF_T)$, we have $I_0(Y)\le 0$ if and only if $ Y=0$. Then, the following statements hold:
  \begin{enumerate}
\item $I_t$ is $\cF_t$-linear, for all $t\in [0,T]$.
\item $I_t$ is increasing, for all $t\in [0,T]$. 
\item $I_t$ is regular, for all $t\in [0,T]$. 
\item The stochastic indicator $I=(I_t)_{t\in [0,T]}$ satisfies the tower property. 
\end{enumerate}
\end{prop}
\begin{proof}
Let us show that  $I_t$ is linear. Consider  $X ,Y \in\cL^0(\R, \cF_T) $, $\alpha \in \R $ and  $F_t \in \cF_t$.
By linearity and the projection  property, we get that
  \bean I_0((\alpha I_t(X)+I_t(Y)) 1_{F_t})&=&  \alpha I_0(I_t(X)  1_{F_t})+I_0(I_t(Y) 1_{F_t})\\
  &=&  \alpha I_0(X  1_{F_t})+I_0(Y 1_{F_t})\\
  &=&I_0((\alpha X+Y) 1_{F_t}).
  \eean 
  By Proposition \ref{uniq1}, we deduce $I_t((\alpha X+Y) )= \alpha I_t(X)+I_t(Y)$. By Theorem \ref{linear}, $I_t$ is then $\cF_t$-linear.   The second statement is an immediate consequence of the first one. The third one   is also direct consequence by Proposition \ref{uniq1}. At last, consider  $X \in\cL^0(\R, \cF_T) $ and let $F_s \in \cF_s$ where $s\le t$. By the projection  property, we have
  $$I_0(I_s(I_t(X)) 1_{F_s})=I_0(I_t(X) 1_{F_s})=
I_0(X1_{F_s}).$$
We conclude by Proposition \ref{uniq1} that $I_s(I_t(X))=I_s(X)$. 
\end{proof}

The conditional expectation $E(X|\cF_t)$, $X\in  \mathbb{L}^{1}(\R_+, \cF_T)$,  is the unique $\cF_t$-measurable random variable such that we have   $E(X1_{F_t})=E(E(X|\cF_t)1_{F_t})$, for all $F_t\in \cF_t$. As soon as a stochastic indicator $I=(I_t)_{t\in [0,T]}$ satisfies the tower property and is such that $I_t$ is $\cF_t$-regular, for all $t\le T$, we have 
 $I_0(X1_{F_t})=I_0(I_t(X1_{F_t}))=I_0(I_t(X)1_{F_t})$, for all $F_t\in \cF_t$. The natural question is whether $I_t(X)$ is the unique $\cF_t$-measurable random variable satisfying the  projection  property. Below, we study the case of the stochastic essential supremum indicator.

\begin{theo} \label{uniq}
Let $X \in \mathbb{L}^{0}(\R_+, \cF_T)$ such that $\esssup_{\cF_0}(X)\in \R $, i.e. $X$ is bounded. There exists a unique $Z_t \in \mathbb{L}^{\infty}(\R_+, \cF_t)$  that satisfies the projection property 
$$\esssup_{\cF_0}(Z_t 1_{F_t}) = \esssup_{\cF_0}(X 1_{F_t}), \forall F_t \in \cF_t.$$
\end{theo}
\begin{proof}  Observe that that $Z_t=ess \sup_{{\cal F}_t}(X)$ exists and is in $\mathbb{L}^\infty(\R_+,{\cal F}_t)$ and satisfies the projection property of the theorem. It remains to prove the uniqueness. So consider another possible 
  $X_t \in \mathbb{L}^{\infty}(\R_+, \cF_t)$ such that 
$$\esssup_{\cF_0}(X_t 1_{F_t}) = \esssup_{\cF_0}(X 1_{F_t}), \forall F_t \in \cF_t.$$

Let us show that $X_t= Z_t$. Consider,  for any $ \e > 0$,
 the set $F_t^{\e}=\{ Z_t \le X_t - \e \}\in \cF_t.$ Then, $X 1_{F_t^{\e}} \le (X_t - \e)1_{F_t^{\e}}$. Moreover, by assumption, we have
 \bean \esssup_{\cF_0}(X_t 1_{F_t^{\e}})=  \esssup_{\cF_0}(X 1_{F_t^{\e}}) 
  \le \esssup_{\cF_0}((X_t - \e)1_{F_t^{\e}}) \le \esssup_{\cF_0}(X_t 1_{F_t^{\e}}).\eean 
Therefore, we have:
$$\esssup_{\cF_0}((X_t - \e)1_{F_t^{\e}}) = \esssup_{\cF_0}(X_t 1_{F_t^{\e}}).$$
 Suppose that $\mathbb{P}(F_t^{\e}) >0$.
We claim that $\esssup_{\cF_0}(X_t 1_{F_t^{\e}})> 0$. Otherwise, $X_t 1_{F_t^{\e}}\le 0$   and so $  X \le -\e$ on $F_t^{\e}$ in contradiction with  $X \ge 0$ a.s.. So, we have  $0< \esssup_{\cF_0}(X_t 1_{F_t^{\e}}) \le \esssup_{\cF_0}(X_t ) \in \R_+$. By Corollary \ref{coro}, $\e=0$ in contradiction with the assumption that $\epsilon>0$. Therefore, $\mathbb{P}(F_t^{\e})=0$ hence $\mathbb{P}(F_t^{1/n})=0$,  for any $n \ge 1$.  We deduce that   $\mathbb{P}(\bigcap_{n \ge 1}  (\Omega \setminus F_t^{1/n}))=1$, which means that, a.s., $X_t-1/n < Z_t$ for any $n \ge 1$. As $n\to \infty$, we get that $X_t \le  Z_t .$ Now consider the sets  $G_t^{\e}=\{ X_t \le Z_t - \e \}.$ Similarly, we obtain that $$\esssup_{\cF_0}((Z_t - \e)1_{G_t^{\e}}) = \esssup_{\cF_0}(Z_t 1_{G_t^{\e}}).$$
 As $\esssup_{\cF_0}(Z_t)\ne 0$ if $\mathbb{P}(G_t^{\e}) >0$, we apply again  Corollary \ref{coro}  and deduce  that $\e =0$, i.e. a contradiction. We deduce that  $Z_t \le  X_t$ a.s. and the conclusion follows.
\end{proof}
\begin{coro}
Let $X \in \mathbb{L}^{\infty}(\R, \cF_T)$ such that $X>0$ a.s.. Then, there exists a unique $Z_t \in \mathbb{L}^{0}(\R, \cF_t)$ such that $\esssup_{\cF_0}(Z_t)\in \R $ that satisfies the projection property
$$\esssup_{\cF_0}(Z_t 1_{F_t}) = \esssup_{\cF_0}(X 1_{F_t}), \forall F_t \in \cF_t.$$
\end{coro}
\begin{proof}
 Suppose that there exists  $Z_t \in \mathbb{L}^{0}(\R, \cF_t)$ such that 
$$\esssup_{\cF_0}(Z_t 1_{F_t}) = \esssup_{\cF_0}(X 1_{F_t}), \forall F_t \in \cF_t.$$

Let us show that $Z_t \ge 0$.
Consider  $F_t=\{ 0 \ge Z_t \}.$ Then, we have $\esssup_{\cF_0}(X 1_{F_t}) = \esssup_{\cF_0}(Z_t 1_{F_t})\le 0$. Therefore, $X 1_{F_t} \le 0$ on $F_t$ hence $\mathbb{P}(F_t)=0$. The conclusion follows by Theorem \ref{uniq}.
\end{proof}

\begin{coro} 
Consider $X \in \mathbb{L}^{\infty}(\R_-, \cF_T)$ (resp. s.t. $X<0$). There exists a unique $Z_t \in \mathbb{L}^{\infty}(\R_-, \cF_t)$ (resp. $Z_t \in \mathbb{L}^{0}(\R_-, \cF_t)$ bounded from below)  that satisfies the projection property
$$\essinf_{\cF_0}(Z_t 1_{F_t}) = \essinf_{\cF_0}(X 1_{F_t}), \forall F_t \in \cF_t.$$  
\end{coro}

The following counter-example shows that uniqueness does not hold in general for the essential supremum indicator.

\begin{ex}
Consider $\Omega= \R$, $\cF= \cB(\R)$ and  $\mathbb{P}$ the probability mesure defined by its density $ d\mathbb{P}/dx= \alpha/(1+ x^2 )$, $\alpha>0$, w.r.t.  the Lebesgue measure $dx$. 
We consider $X=0$  and we define $Z(\omega)=-exp(w)$ for all $\omega \in \Omega$. \smallskip

We claim  that $0=\esssup_{\cF_0}(Z)$. First,  as  $Z \le 0$, $\esssup_{\cF_0}(Z) \le 0$.  Secondly, as $\lim_{\omega \rightarrow -\infty} Z(\omega)=0$, for any $ \alpha <0$, there exists $ x \in \R$ such that for any $ \omega \le x $, $ 0 \ge Z(\omega) \ge \alpha$. So  $0 \ge \esssup_{\cF_0}(Z) \ge \alpha$ for any $\alpha \in \R^-$. Therefore, $\esssup_{\cF_0}(Z) =0$. \smallskip

 Let $A= \{ Z \le -1 \}$ and   $\cF_1 = \sigma (1_A ) $, i.e. $ \cF_1 = \{ A, A^c, \Omega, \emptyset \}$. 
Let us introduce $Z_1=\esssup_{\cF_1}(Z)$. Note that $Z_1 1_A= \esssup_{\cF_1}(Z 1_A) =-1_A$ and $Z_1 1_{A^c}= \esssup_{\cF_1}(Z 1_{A^c}) =0$, i.e.   $Z_1 =-1_A$.  \smallskip

Let us now consider any $F_1 \in \cF_1$. Of course, $\esssup_{\cF_0}(X 1_{F_1}) =0$. On the other hand, we have by the tower property and by sub-additivity: \bean 0=\esssup_{\cF_0}(Z)&=&\esssup_{\cF_0}(Z_1)\\
 &\le& \esssup_{\cF_0}(Z_1 1_{F_1}) + \esssup_{\cF_0}(Z_1 1_{F_1^c}) \le 0. \eean
 We deduce that $\esssup_{\cF_0}(Z_1 1_{F_1})= \esssup_{\cF_0}(Z_1 1_{F_1^c}) = 0$.  Therefore,
  $$\esssup_{\cF_0}(Z_1 1_{F_1}) =0= \esssup_{\cF_0}(X 1_{F_1}),\, \forall F_1 \in \cF_1.$$
 
 This means that the projection property for $X=0$ is both satisfied by $\esssup_{\cF_0}(X)=0$ and $Z_1$, i.e. uniqueness does not hold. Note that, with  $\cF_2 = \cB(\R)$ and  $Z_2=\esssup_{\cF_2}(Z)=Z$, we have, for any   $F_2 \in \cF_2$,
 $\esssup_{\cF_0}(X 1_{F_2}) =0= \esssup_{\cF_0}(Z_2 1_{F_2})$. However, $Z_2=Z \neq 0$.  
\end{ex}

\section{Risk measures derived from  conditional indicators}\label{SecRM}

Let $I$ be a C.I. w.r.t. a $\sigma$-algebra $\cH$. We define the positive elements of $I$ as the set 
$$\eD_I^+:=\{X\in \eD_I:~I(X)\ge 0\}.$$
In the setting of risk measures in finance, the elements of $\eD_I^+$ are interpreted as the acceptable financial positions. We then define:
$$\cM_I(X):=(\eD_I^+-X)\cap \eL(\R,\cH),\quad X\in \eL(\overline{\R},\cF).$$
Note that $\cM_I(X)$ may be empty and we have $\cM_I(X+\alpha_{\cH})=\cM_I(X)-\alpha_{\cH}$ for all $\alpha_{\cH}\in \eL(\R,\cH)$. Moreover, $Y_{\cH}\in \cM_I(X)$ if and only if  $Y_{\cH}+X$ is acceptable. We then define:

\bea \label{RMDef} \rho_I(X):=\essinf \cM_I(X),\,X \in \eL(\R,\cF),\eea
with the convention $\essinf \emptyset=+\infty$. Here, we use the usual notation $\essinf \Gamma$ without mentioning the $\sigma$-algebra when this one is shared with the elements of  $\Gamma$, i.e. $\cH=\cF$ in the definition.  We denote by ${\rm Dom\,} \rho_I $ the set of all $X  \in \eL(\R,\cF)$, such that $\cM_I(X) \neq \emptyset $.  \smallskip

The proof of the following lemma, being simple, is left to the readers.

\begin{lemm} \label{DOM}
Let  $I$ be  a conditional operator w.r.t. $\cH$. We have the following properties: 
\begin{enumerate}
\item If $I$ is $\cH$-positively homogeneous, then $\alpha_{\cH} X \in  {\rm Dom\,} \rho_I $ for every $X$ in ${\rm Dom\,} \rho_I $ and $\alpha_{\cH} \in  \cL^0 (\R_+ , \cH )$. 
\item If $I$ is super-additive, then $X_1 + X_2 \in{\rm Dom\,} \rho_I $ for any $X_1, X_2 \in {\rm Dom\,} \rho_I $. 
\item If $I$ is non decreasing, then   $X_1 \in {\rm Dom\,} \rho_I $ for any $X _1 \ge X_2 $ such that $ X_2 \in {\rm Dom\,} \rho_I $.
\item If $\eD_I^+$ is $\cH$-convex (e.g. if $I$ is $\cH$-convex) then ${\rm Dom\,} \rho_I$  is $\cH$-convex. 

\end{enumerate}
\end{lemm}

We now recall the definition of a risk measure, see \cite{FP} for example.

\begin{defi}

Let $\eD$ be a subset of $\eL(\R, \cF)$ containing $0$ and such that $\eD+ \eL(\R, \cH)\subset \eD$.
We say that a mapping
\bean
\rho_{\cH}: \eD & \longrightarrow & \eL(\overline{\R}, \cH).\\
 X & \longmapsto & \rho_{\cH}(X) 
\eean
is an $\cH$-conditional risk measure  if the following properties hold:
\begin{itemize}
\item[(P1)] Normalization: $\rho_{\cH}(0)=0$. 
  \smallskip
\item[(P2)]  Monotonicity: $\rho_{\cH} (X_1) \le \rho_{\cH}(X_2)$, for any $X_1, X_2 \in \eD$ such that we have $X_1 \ge X_2$.
 
\item[(P3)] Cash invariance: for all $ X \in \eD$ and $\alpha_{\cH} \in \eL(\R,\cH)$, we have the equality  $\rho_{\cH}(X+ \alpha_{\cH}) = \rho_{\cH}(X) - \alpha_{\cH} $.
\end{itemize}

\end{defi}

Recall that we say that  $\eD\subseteq L^0(\R,\cF)$ is $\cH$-convex if, for all $ X_1,X_2 \in \eD$ and $\alpha_{\cH} \in \eL([0,1],\cH)$, we have $\alpha_{\cH} X_1 + (1-\alpha_{\cH}) X_2 \in \eD$. When $\alpha_{\cH} X \in \eD$ for all  $ X\in \eD$ and $\alpha_{\cH} \in \eL(\R_+,\cH)$, we say that $\eD$ is positively homogenous. 

 \begin{defi}
 An $\cH$-conditional risk measure $\rho$ on $\eD_{\rho}$ is said:
 \begin{enumerate}
 
 \item conditionally convex if $\eD_{\rho}$ is $\cH$-convex and, for all $ X_1,X_2 \in \eD_{\rho}$ and $\alpha_{\cH} \in \eL([0,1],\cH)$, we have: $$\rho(\alpha_{\cH} X_1 + (1-\alpha_{\cH}) X_2)\le \alpha_{\cH} \rho( X_1) + (1-\alpha_{\cH})\rho( X_2).$$   
   
 \item conditionally positively homogeneous if $\eD$ is positively homogenous and,  for all $ X \in \eD_{\rho}$ and $\alpha_{\cH} \in \eL(\R_+,\cH)$, $ \rho( \alpha_{\cH} X) =  \alpha_{\cH} \rho( X)$.   
  \end{enumerate}

  A conditional convex risk measure which is positively homogeneous is called a conditional coherent risk measure.
 \end{defi} 

The proof of the following lemma is standard:

\begin{prop}\label{RM}
Let  $I$ be  a non decreasing conditional operator. Consider  the mapping $\rho_I$ defined by (\ref{RMDef}) and the associated domain ${\rm Dom}\, \rho_I$. We have  the following properties:
\begin{enumerate}

\item The mapping $\rho_I$ is a conditional risk measure on ${\rm Dom} \rho_I$.

\item If $\eD_I^+$ is $\cH$-convex (for example if $I$ is $\cH$-convex), then $\rho_I$  is $\cH$-convex. 
\item If $I$ is $\cH$-positively homogeneous, then $\rho_I$ is  $\cH$-positively homogeneous. 
\item If $I$ is super-additive, then $\rho_I$ is sub-additive. 
\end{enumerate}
\end{prop} \label{RM}

  \begin{lemm}
 Let us consider a C.I. $I$ and let us define $\rho(X)=I(-X)$. Then, $\rho$ is a conditional risk-measure if and only if $I$ is increasing and $\cH$-translation invariant. 
 \end{lemm}
 
 \begin{lemm}
 
  Let us consider a C.I. $I$ and let us define $\rho(X)=-I(X)$. Then, $\rho$ is a conditional risk-measure if and only if $I$ is increasing and $\cH$-translation invariant. 
 \end{lemm}
 
 Under some conditions, we may show that a risk-measure admits a dual representation at least on $L^1(\R,\cF_T)$, see \cite{FP} and the recent result on $L^0(\R,\cF_T)$ in  \cite{LV1}. The open question is whether a conditional indicator may have such a characterization, at least if it is convex.

\section{The  conditional expectation indicator}\label{CEI}
 The  conditional expectation knowing $\cH\subseteq \cF$ is defined on  $\eL(\R, \cF)$ with the conventions introduced in the beginning of the paper by:
  $$\E(X \vert \cH):=\E(X^+\vert \cH)-\E(X^-\vert \cH).$$
  
  \begin{lemm} \label{LemmCondExp} The mapping $I(X)= \eE(X\vert \cH)$ is a conditional indicator on $X\in \mathbb{D}_{I}= \eL(\R,\cF)$. Moreover, if $X\in \eL(\R,\cF)$,  we have:
\bea \label{RestOper}
1_{H}I(X)&=&I(X1_{H}),\,{\rm for\,all\,} H\in \cH,\\
\label{LinExp}
I(X+\alpha_{\cH})&=&I(X)+\alpha_{\cH}, {\rm on\, the\,set\,} \{ \E(X^+\vert \cH)\ne +\infty \,{\rm or \,} \E(X^-\vert \cH)\ne +\infty \},\quad \quad \\ \label{LinMult}
I(\alpha_{\cH}X)&=&\alpha_{\cH}I(X),\,{\rm for\, all\,} \alpha_{\cH}\in \eL(\R,\cH).
\eea

\end{lemm}
\begin{proof} Note that $(X1_H)^+=1_HX^+$ and $(X1_H)^-=1_HX^-$ so that (\ref{RestOper}) holds by the property satisfied by the usual conditional expectation defined on the non negative random variables. Therefore, we may show the next properties on each subset of  a  $\cH$-measurable partition of $\Omega$. Actually, the equality  (\ref{LinExp}) is a particular case of Lemma \ref{CondExpt-Linear} we show below.

To show (\ref{LinMult}), it suffices to consider the cases $\alpha_{\cH}\ge 0$ and $\alpha_{\cH}<0$. When, $\alpha_{\cH}<0$, we have $(\alpha_{\cH}X)^+=|\alpha_{\cH}|X^-$ and $(\alpha_{\cH}X)^-=|\alpha_{\cH}|X^+$. We deduce that $\E(\alpha_{\cH}X\vert \cH)=|\alpha_{\cH}|\E(X^-\vert \cH)-|\alpha_{\cH}|\E(X^+\vert \cH)$ and we conclude  that $\E(X\vert \cH)=-\alpha_{\cH}\E(X^-\vert \cH) +\alpha_{\cH}\E(X^+\vert \cH)=\alpha_{\cH}\E(X\vert \cH)$. \smallskip

At last, $X\le \alpha_{\cH}=\esssup_{\cH}(X)$ and it is clear that $\E(X\vert \cH)\le \alpha_{\cH}$ when $\alpha_{\cH}=+\infty$. Otherwise, as $\E(X^+\vert \cH)\le (\esssup_{\cH}(X))^+$, $\alpha_{\cH}\ne +\infty$ implies that $\E(X^+\vert \cH)\ne +\infty$ hence (\ref{LinExp}) applies. So, $\E(X-\alpha_{\cH}\vert \cH)=\E(X\vert \cH)-\alpha_{\cH}$. As, $X-\alpha_{\cH}\le 0$, we get that $\E(X-\alpha_{\cH}\vert \cH)\le 0$ hence 
$\E(X\vert \cH)-\alpha_{\cH}\le 0$ and, finally, $\E(X\vert \cH)\le \alpha_{\cH}$. Indeed, it suffices to observe that $\E(X\vert \cH)\ne +\infty$.

As $X\ge \alpha_{\cH}=\essinf_{\cH}(X)$, we conclude similarly and the conclusion follows.
\end{proof}

\begin{lemm}\label{CondExpt-Linear}
 Let us consider the operator $I(X)= \eE(X\vert \cH)$  on $ \mathbb{D}_{L}= \eL(\R,\cF)$. If $X, Y\in \eL(\R,\cF)$,  then we have $I(X+Y)=I(X)+I(Y)$  on the set $F=\bigcup _{i\in [1,5]} F_i\in\cH$ where
 
\bean F_1&=&\{\E(|X|\vert \cH),\E(|Y|\vert \cH)\in \R\},\\
F_2&=&\{\E(X^+\vert \cH)=+\infty, (\E(X^-\vert \cH) ,\E(Y^-\vert \cH))\in \R^2\},\\
F_3&=&\{\E(X^-\vert \cH)=+\infty, (\E(X^+\vert \cH),\E(Y^+\vert \cH))\in \R^2\},\\
F_4&=&\{\E(Y^+\vert \cH)=+\infty, (\E(X^-\vert \cH) ,\E(Y^-\vert \cH))\in \R^2\},\\
F_5&=&\{E(Y^-\vert \cH)=+\infty, (\E(X^+\vert \cH),\E(Y^+\vert \cH))\in \R^2\}.
\eean
\end{lemm}
\begin{proof} In the following, we shall use the convexity and positive homogeneity of the mappings $x\mapsto x^+=\max(x,0)$ and $x\mapsto x^-=\max(-x,0)$. This implies that $(x+y)^+\le x^++y^+$ and $(x+y)^-\le x^-+y^-$ for all $x,y\in \R$. Consider $X,Y\in \eL(\R,\cF)$. We then have:
\bea \label{IneqAuxLemmCondExp} 
&&\E((X+Y)^+\vert \cH)\le \E(X^+\vert \cH)+ \E(Y^+\vert \cH),\\\nonumber
&&  \E((X+Y)^-\vert \cH)\le \E(X^-\vert \cH)+\E(Y^-\vert \cH),\\ \nonumber
&&\E(X^+\vert \cH)\le \E((X+Y)^+\vert \cH)+\E(Y^-\vert \cH),\\
&& \E(X^-\vert \cH)\le \E((X+Y)^-\vert \cH)+\E(Y^+\vert \cH). \nonumber
\eea

Note that the inequalities above are obvious as soon as one of the terms in the r.h.s. is $+\infty$. Otherwise, we may argue as if the random variables were integrable. \smallskip

\noindent 1rst case: On the set $F_1$, $X$ and $Y$ are integrable so the result  holds by linearity of the conditional expectation for integrable random variables. \smallskip

\noindent 2nd case: On the set $F_2$, by (\ref{IneqAuxLemmCondExp}), we have $\E((X+Y)^-\vert \cH)\in \R$ and $\E((X+Y)^+\vert \cH)=+\infty$. Therefore, $L(X+Y)=L(X)+L(Y)=+\infty$ and the equality holds.\smallskip

\noindent 3rd case: On the set $F_3$, by (\ref{IneqAuxLemmCondExp}), we have $\E((X+Y)^-\vert \cH)=+\infty$ and $\E((X+Y)^+\vert \cH)\in \R$. Therefore, $L(X+Y)=L(X)+L(Y)=-\infty$.\smallskip

By symmetry, the same conclusion holds on the  subsets $F_4$ and $F_5$.  \end{proof}

\begin{prop}
Let $\mathcal{H}$ be a sub $\sigma$-algebra of $\mathcal{F}$  and let $I$ be a C.I. w.r.t. $\mathcal{H}$ such that $\eD_{I}=L^1(\R,\cF)$ and
\begin{itemize}

\item [1)] $I^*=I$
 \item [2)] $I$ is sub-additive (respectively super-additive).
 \item [3)] $E(|I(X)| )\le  \E( |X| )$.
\end{itemize}
Then $I= \E(. | \mathcal{H})$ on $L^1(\R,\cF)$.
\end{prop}
\begin{proof}
By Proposition \ref{linear}, $I$ is a linear indicator. As $1 \in \eL(\R,\mathcal{H})$,  $I(1)= 1$. Moreover, $I$ is contractive by assumption. We conclude by Douglas Theorem, see   \cite{AAB}, that  $I= \E(. |G)$ where $G$ is the $\sigma$-algebra generated by the fixed points  of $I$.  It is clear that $G={\mathcal{H}}$  hence $I= \E(. |{\mathcal{H}})$. 
\end{proof}

\begin{theo}\label{main}

Suppose that  $I$ is a conditional indicator defined on the domain $\eD_I=\mathbb{L}^{1}(\R , \mathcal{F})$  and satisfies the following properties: 
\begin{itemize} 

\item [1)] $I$ is self-dual.

\item [2)] $I(X+Y)= I(X)+I(Y)$ for all $X ,Y\in \mathbb{L}^{1}(\R , \mathcal{F}) $. 
 
  \item [3)]  $I$ satisfies the Fatou property i.e.,  for any sequence $(X_n)_n$ of  $\mathbb{L}^{1}(\R , \mathcal{F})$, we have $I(\liminf_n X_n) \leq  \liminf I(X_n) $.
\end{itemize}
Then, there exists a probability measure $\mu <<\mathbb{P}$, with $\rho = d\mu / d\mathbb{P} \in \mathbb{L}^{1}(\R_+, \mathcal{F})$ such that  $ I(X)=E_\mu(  X \vert \mathcal{H})=\E(\rho X | \mathcal{H})$, for all $X \in  \mathbb{L}^{1}(\R , \mathcal{H}). $
\end{theo}

\begin{proof}
Since $L$ is sub-additive and self-dual, we deduce   
by Lemma \ref{linear} that $L$ is $\mathcal{H}$-linear.  Moreover,  $L$ is increasing by Lemma \ref{pos}. Let us define the mapping $\mu(A) = \E( I(1_A)) $, for any $A \in \mathcal{F}$. Let us prove that $\mu$ is a probability measure. 

As $1_A \in [0,1]$, then $I(1_A) \in [0,1]$ a.s. hence $\mu(A)\in [0,1]$. Moreover, $I(1_\Omega) =I(1)=1$. So $\mu(\Omega)=1$. Also $\mu(\emptyset)= \E( I(0))=0$. Consider a partition $(A_n)_{n\in \N}$  of $\Omega$. We have $$\sum_{n=0} ^\infty 1_{A_n}= \lim_{N\rightarrow\infty } \sum_{n=0} ^N 1_{A_n} \ge \sum_{n=0} ^N 1_{A_n}.$$
Therefore,
 \bean I\left(\sum_{n=0} ^\infty 1_{A_n}\right)  \ge  \lim_{N\rightarrow\infty } I\left( \sum_{n=0} ^N 1_{A_n}\right) =\lim_{N\rightarrow\infty }  \sum_{n=0} ^N I( 1_{A_n}).
 \eean
  
 We deduce that $\mu \left(\sum_{n=0} ^\infty 1_{A_n}\right) \ge  \ \sum_{n=0} ^\infty \mu ( 1_{A_n})$. Moreover, by the Fatou property, we have
 \bean I\left( \lim_{N\rightarrow\infty } \uparrow \sum_{n=0} ^N 1_{A_n}\right) \le  \lim_{N\rightarrow\infty } \uparrow I\left( \sum_{n=0} ^N 1_{A_n}\right) =\lim_{N\rightarrow\infty }  \sum_{n=0} ^N I( 1_{A_n}). \eean
  So, $\mu (\sum_{n=0} ^\infty 1_{A_n}) \le  \ \sum_{n=0} ^\infty \mu ( 1_{A_n})$. We conclude that $\mu$ is a probability measure. Note that, if $\mathbb{P}(N) = 0$, then $I( 1_{N})=I( 0)=0$, i.e. $\mu(N)=0$. Therefore $\mu$ is absolutely continuous w.r.t. $P$. Let $ \rho= d\mu/ dP $ be the Radon-Nikodym derivative. We aim to show that $I=\E_\mu [.|H]$.

Consider  $A \in  \mathcal{H}$. In one hand,  $ \mu(A) =\E( I(1_A)) = \E(1_A)$ by definition of $\mu$ and $I$. In the other hand, $\mu(A) =\E( \rho 1_A)$ as $ \rho= d\mu/ d\mathbb{P}$. Therefore, $\E( \rho 1_A)=\E(  1_A)$ for any $A \in  \mathcal{H}$ hence $\E( \rho | \mathcal{H})=1$. Moreover, for any $A \in \mathcal{F}$, as  $I(1_A)$ is $\mathcal{H}$-measurable, we get that
$$\E_\mu(I(1_A))= \E( \rho I(1_A))= \E(I(1_A))= \mu(A)= \E_\mu(1_A).$$
So, for any $A \in \mathcal{F}$ and $B \in \mathcal{H}$, we have:
$$\E_\mu(1_A 1_B)= \E_\mu(I(1_A 1_B))=\E_\mu(1_BI(1_A )).$$
This implies that $ \E_\mu[1_A | \mathcal{H}]= I(1_A)$,  for all $A \in \mathcal{F}.$

Consider now $X \in \mathbb{L}^{0}(\R_+, \mathcal{F}).$ We use the standard arguments, i.e. we have $X=  \lim_n\uparrow X_n$   where   $X^n=\sum_{i=0} ^n \alpha_i^n 1_{A_i^n}$, $(A_i^n)_i$ is a partition of $\Omega$ in $\cF$ and $\alpha_i^n\in \R$.  By the Fatou property,  $I(X)\le \lim_{n}\uparrow \sum_{i=0} ^n  \alpha_i^n  I(1_{A_i^n}).$ 
 On the other hand, $X \ge  \sum_{i=0} ^n \alpha_i^n 1_{A_i^n}$. implies that  $I(X)\ge  \sum_{i=0} ^n  \alpha_i^n  I(1_{A_i^n})$ for any $n$. Therefore, $I(X)\ge \lim_{n}\uparrow \sum_{i=0} ^n  \alpha_i^N  I(1_{A_i^n})$. We deduce that:
  $$I(X) =  \lim_{n}\uparrow \sum_{i=0} ^n  \alpha_i^n  I(1_{A_i^n})=  \lim_{n}\uparrow \sum_{i=0} ^n  \alpha_i^n \E_\mu [ 1_{A_i^n}|\mathcal{H}]= \E_\mu [ X|\mathcal{H}].$$  

Finally,  for any $X \in \mathbb{L}^{1}(\R, \mathcal{F})$,  we have:
$$I(X)=I(X^+)-I(X^-)=  \E_\mu [ X^+|\mathcal{H}]- \E_\mu [ X^-|\mathcal{H}]= \E_\mu [ X|\mathcal{H}].$$ 
   Since  $\E(\rho |H)=1$, we finally deduce that  $I(X)  =\E(\rho X | \mathcal{H}).$
 \end{proof}
 
 In the following, we construct linear indicators that are not conditional expectations. \smallskip

\noindent {\bf Counter-example} The following is standard. Consider the space $\Omega=\N$ of all non negative integers endowed with the $\sigma$-algebra $\cF$ of all subsets of $\N$. The probability measure is defined as $P(A)=\sum_{n=0}^\infty
 2^{-n-1}\delta_n(A)$ where $\delta_n$ is the Dirac measure at point $n$. Let $\cF_0=\{\emptyset,\Omega\}$ be the trivial sub $\sigma$-algebra. Each random variable w.r.t. $(\Omega,\cF,P)$ is identifiable with the sequence $(X(n))_{n\in \N}$ and we have $X_k$ converges a.s. to $X$ when $k\to \infty$ if and only if $X_k(n)\to X(n)$, for all $n\in \N$. The $L^\infty$ norm is $\|X\|_{\infty}=\sup_n|X(n)|$ and we have $\esssup_{\cF_0}=\sup_nX(n)$ and $\essinf_{\cF_0}=\inf_nX(n)$. In the following, we consider the set $\eD$ of all $X\in L^0(\R,\cF)$ such that $\lim_n X(n)$ exists in $\R$. We define the linear positive operator $T(X)=\lim_n X(n)$ on the domain $\eD$. Note that $X\in  L^0(\R,\cF_0)$ if and only if $X$ is a constant sequence so that $X\in \eD$. In particular, $T(X)=X$ for any $X\in  L^0(\R,\cF_0)$ and $T$ is a C.I. w.r.t. $\cF_0$. As $|T(X)|\le \|X\|_{\infty}$ for all $X\in  L^\infty(\R,\cF)$, the Hahn-Banach theorem states the existence of a (continuous) linear mapping $\bar T$ defined on the whole space $L^\infty(\R,\cF)\supset \eD$ such that $\bar T=T$ on $\eD$, see \cite{Bourb}[Chapter 2, 3].
 
 For every $A\in \cF$, consider the random variable
 $$X^A(\omega)=\frac{{\rm card}(A\cap [0,n])}{n+1},$$ where ${\rm card}$ designates the number of elements that contains a subset. We have $X^A\in  L^\infty[0,1],\cF)$. Let us define $m(A)=\bar T(X^A)$. We may show that $m(\emptyset)=0$, $m(\Omega)=1$ and $m(\sum_{i=1}^nA_i)=\sum_{i=1}^n m(A_i)$ if $(A_i)_{i=1}^n$ is a finite partition. 
 
 Moreover, suppose that $\bar T$ is an expectation on $L^\infty(\R,\cF)$, i.e. there exists an integrable random variable $Y\in L^1(\R_+,\cF)$ such that $\bar T(X)=E(YX)$ for any $X\in L^\infty(\R,\cF)$. In that case, if $(A_i)_{i=1}^n$ is an infinite partition, we get that 
 $$1=m(\Omega)=\sum_{i=1}^\infty m(A_i).$$
With $A_i=\{i\}$, $i\ge 0$, we get that $\hat T(X^{A_i})=T(X^{A_i})=0$, i.e. $m(A_i)=0$ for all $i\in \N$.  This is in contradiction with the equality above.

Another example is to consider $\cF_1=\{\emptyset,\Omega,I,I^c\}$ where $I^c=\Omega\setminus I$ and $I=2\N+1$. Le us introduce the indicator:
\bean T_1(X)&=&\hat T(\tilde X) 1_I+E(X|\cF_1)1_{I^c},\quad X\in L^\infty(\R,\cF),\\
\tilde X(n)&=&X(n)1_{I}(n)+X(n+1)1_{I^c}(n).
\eean
We observe that $T_1$ is linear, $T_1(X)$ is $\cF_1$-measurable for any $X\in L^\infty(\R,\cF)$, i.e.  $T_1(X)$ is constant on $I$ and $I^c$ respectively and, moreover, $T_1(X)=X$ if $X\in L^\infty(\R,\cF)$. Therefore, we deduce by monotony that $T_1$ is a linear conditional indicator. By the same arguments, we then prove that $T_1$ is not a conditional expectation of the form 
$T_1(X)=E(XY|\cF_1)$, $X\in L^\infty(\R,\cF)$, for some $Y\in L^\infty(\R,\cF)$. Indeed, otherwise, we get the equality $1=0$ on the non null set $I$.

\section{Appendix}
\begin{lemm}\label{Q}
For $r_H \in \mathbb{L}^{1}(\Q, \mathcal{H})$. 
     There exist $p_H, q_H  \in \mathbb{L}^{1}(\mathbb{Z}, \mathcal{H}) \times \mathbb{L}^{1}(\N^*, \mathcal{H})$ such that $r_H= p_H/ q_H$.
\end{lemm}
\begin{proof}
Consider the random set
$\Gamma( \omega) = \{(p,q) \in \mathbb{Z} \times \N^*: r_H q= p\}$. We observe that 
  its graph $Graph \Gamma=\{(\omega,p,q):~(p,q)\in \Gamma(\omega)\}$ is a measurable set of $\mathcal{H} \times \cB(\mathbb{Z}) \times \cB(\N^*)$, $\sigma$-algebra product of $\cH$ and the Borel  $\sigma$-algebras of $\mathbb{Z}$ and $\N^*$ respectively and $\Gamma( \omega)$ is non empty. Indeed $f :\omega \mapsto r_H(\omega) q- p$ is measurable. Therefore, we conclude by a measurable selection argument, see \cite{KS}[Section A.4], that there exists a measurable selector $(p_H, q_H)$ of $\Gamma$.
\end{proof}

\begin{lemm}\label{R}
 For any  $\alpha \in \mathbb{L}^{1}(\R, \mathcal{H})$, there exist $(r_n)_n , (q_n)_n \in \mathbb{L}^{1}(\Q, \mathcal{H})$ such that $$\alpha-1/n \le r_n \le \alpha \le q_n \le \alpha+1/n . $$
\end{lemm}
\begin{proof}
Consider $\Gamma( \omega) = \{(r_n,q_n) \in \mathbb{Q}^2 : \alpha-1/n \le r_n \le \alpha \le q_n \le \alpha+1/n \}$. It is non empty a.s. and its graph is a measurable set of $\mathcal{H} \times \cB(\mathbb{Z}) \times \cB(\N^*)$. We then conclude by a measurable argument, see \cite{KS}[Section A.4].
\end{proof}

\begin{lemm}\label{EssupProj}
Consider $F_t \in \cF_t$ such that $ \mathbb{P}(F_t) >0$ and $X \in \mathbb{L}^{0}(\R, \cF_T)$ such that $X=X 1_{F_t}$ and $\esssup_{\cF_0}(X)\ne 0$.
If, for some $\e\in \R_+$, we have  \bea \label{EssupProj-1} \esssup_{\cF_0}(X- \e  1_{F_t})=\esssup_{\cF_0}(X), \eea
  then $\e=0$. 

\end{lemm}
\begin{proof}
A first  case is when $ 1_{F_t}=1$. In that case $\esssup_{\cF_0}(X)- \e =\esssup_{\cF_0}(X)$, thus $\e=0$. 
Suppose that $ 1_{F_t} < 1$, i.e. $ \mathbb{P}(\Omega \setminus F_t) >0$.  As $\esssup_{\cF_0}( X)\ge X$ a.s., we deduce that 
$\esssup_{\cF_0}( X)\ge 0$ on $\Omega \setminus F_t\ne \empty$ hence $\esssup_{\cF_0}( X)\ge 0$. Note that (\ref{EssupProj-1}) is equivalent to     $\alpha_0:=\esssup_{\cF_0}(X+ \e  1_{\Omega \setminus F_t})=\esssup_{\cF_0}(X)+ \e $. We observe that $\alpha_0\ge \esssup_{\cF_0}( X) \vee \e$ and, also,  $\esssup_{\cF_0}( X) \vee \e \ge X + \e  1_{\Omega \setminus F_t}$ on $\Omega \setminus F_t$. \smallskip

On the other hand, on the set $F_t$, we also have
 $$ \alpha_0\ge \esssup_{\cF_0}( X) \vee \e \ge X = X + \e  1_{F_t^c}.$$ 
 Therefore, a.s. we have $\alpha_0\ge \esssup_{\cF_0}( X) \vee \e \ge \esssup_{\cF_0}(X+ \e  1_{\Omega \setminus F_t}). $
 We deduce that $\esssup_{\cF_0}( X) \vee \e =\alpha_0.$ This implies that
(\ref{EssupProj-1}) is equivalent to $$\esssup_{\cF_0}( X) \vee \e =\esssup_{\cF_0}( X) + \e.$$ 
If $\esssup_{\cF_0}( X) \ge  \e$, then $\esssup_{\cF_0}( X) =\esssup_{\cF_0}( X) + \e$ and thus $\e =0$. If $\esssup_{\cF_0}( X) < \e$, we get that $\esssup_{\cF_0}( X)=0$ in contradiction with the assumption.
\end{proof}
\begin{coro} \label{coro}
Consider $F_t \in \cF_t$ such that $ \mathbb{P}(F_t) >0$ and $X \in \mathbb{L}^{0}(\R, \cF_T)$ such that $\esssup_{\cF_0}(X 1_{F_t}) \ne 0$.
If, for some $\e \ge 0$,  $$\esssup_{\cF_0}((X- \e)  1_{F_t})=\esssup_{\cF_0}(X 1_{F_t}),$$
  then $\e=0$. 

\end{coro}

All authors contributed to the study conception and design. Material preparation, data collection and analysis were performed by Dorsaf Cherif  and Emmanuel Lépinette. The first draft of the manuscript was written by Dorsaf Cherif  and Emmanuel Lépinette. All authors read and approved the final manuscript.


\begin{thebibliography}{100}




\bibitem{AAB} Abramovich Y.A., Aliprantis C.D. and Burkinshaw O. An elementary proof of Douglas'theorem on contractive projections on $L^1$-spaces. Journal of Mathematical Analysis and Applications, 177, 2, 641-644, 1993.

 \bibitem{ABCM}
Azouzi Y., Ben Amor M.A, Cherif D., Masmoudi M. Conditional supremum in Riesz spaces. Journal of Mathematical Analysis and Applications, 531, 1, 2, 2024.

\bibitem{AL} Aliprantis  C.D. and Burkinshaw O. Positive operators. Springer Science and Business Media, 2006.

\bibitem{BCJ} Barron E.N., Cardaliaguet P. and Jensen R. Conditional essential suprema with applications. Applied Mathematics and Optimization, 48, 229-253, 2003.

\bibitem{Bourb} Bourbaki N. Espaces vectoriels topologiques, Chapitres 1 à 5. Ed. 1,  Springer Berlin, Heidelberg 2007.

\bibitem{CL} Carassus L. and Lepinette E. Pricing without no-arbitrage condition in discrete-time.  Journal of Mathematical Analysis and Applications, 505, 1, 2022. 

\bibitem{DelSch05} Delbaen F. and  Schachermayer W. The mathematics of arbitrage. Springer Finance, 2006.

\bibitem{Delb1}  Delbaen F. Coherent risk measures on general probability spaces. Advances in Finance and Stochastics: essays in honor of Dieter Sondermann, Springer, Heidelberg, 1-37, 2002.

 

\bibitem{DS}   Detlefsen K. and  Scandolo G. Conditional and dynamic convex risk
measures. Finance and Stochastics, 9, 539-561, 2005.



\bibitem{DMW} Dalang E.C., Morton A. and  Willinger W. Equivalent martingale measures and no-arbitrage in stochastic securities market models.  Stochastics and Stochastic Reports,  29, 185-201, 1990.

 \bibitem{EL} El Mansour M. and Lepinette E. Conditional interior and conditional closure of random sets. Journal of Optimization Theory and Applications, 187, 356-369, 2020.
 
 \bibitem{EL1} El Mansour M. and Lepinette E. Robust discrete-time super-hedging strategies under AIP condition and under price uncertainty. MathematicS In Action,11, 193-212, 2022.
  

\bibitem{FollS}F\"ollmer H. and A. Schied. Stochastic Finance: An introduction in discrete time. 2nd. Ed., de Gruyter Studies in Mathematics, Walter de Gruyter, Berlin-New York, 2004.

\bibitem{FP} F\"ollmer H. and Penner I. Convex risk measures and the dynamics of their penalty functions. Statistics and Decisions, 24, 61-96, 2006.

\bibitem{G} Grobler J.J. Continuous stochastic processes in Riesz spaces: the Doob–Meyer decomposition. Positivity, 14, 731-751, 2010.

\bibitem{HK79} Harrison J.M. and Kreps D.M.   Martingale and arbitrage in multiperiods
securities markets. Journal of Economic Theory, 20, 381-408, 1979.

\bibitem {KLW} Kuo W.C.,  Labuschagne C.C.A., and Watson W.A. Watson. Discrete-time stochastic processes on Riesz spaces. Indagationes Mathematicae 15.3, 435-451, 2004.

\bibitem{KL} Kabanov Y. and Lepinette E. Essential supremum with respect to a random partial order. Journal of Mathematical Economics, 49, 6, 478-487, 2013.   

\bibitem{KS} Kabanov Y. and Safarian M. Markets with transaction costs. Mathematical Theory, Springer-Verlag, 2009.

\bibitem{Lee} Lee H. Value at Risk.  Risk Management, Springer Texts in Business and Economics, Springer, 75-87,  2021.

\bibitem{LT} Lepinette E. and Tran T. Arbitrage theory for non convex financial market models.  Stochastic Processes and Applications, 127, 10, 3331-3353, 2017. 

 \bibitem{LV} Lepinette E. and Vu D.T. Dynamic programming principle and computable prices in financial market models with transaction costs. Journal of Mathematical Analysis and Application, 524, 2, 2023.

\bibitem{LV1} Consistent Risk Measure on $L^0$ : NA condition, pricing and dual representation. IJTAF,  24, 2022.

\bibitem{LZ} Lepinette E. and Zhao Jun. Super-hedging a European option with a coherent risk-measure and without no-arbitrage condition, Stochastics, 2022.





\bibitem{PenMOR} Pennanen T. Convex duality in stochastic optimization and mathematical Finance.  Mathematics of Operations Research, 36, 2,  340-362, 2011.

\bibitem{Schal99} Schal M. Martingale measures and hedging for discrete-time financial markets.  Mathematics of Operations Research, 24,  509-528, 1999.

\end{thebibliography}
\end{document}